\documentclass[12pt]{amsart}
\usepackage{amssymb,latexsym}
\usepackage{amsfonts}
\usepackage{amsmath}
\usepackage[colorlinks,linkcolor=blue,anchorcolor=blue,citecolor=blue]{hyperref}
\usepackage{algorithm}
\usepackage{enumerate}
\usepackage{algpseudocode}
\usepackage{verbatim}
\usepackage{graphicx}

\newcommand\F{{\mathbb F}}
\newcommand\Q{{\mathbb Q}}

\newcommand\Z{{\mathbb Z}}

\newtheorem{theorem}{Theorem}[section]
\newtheorem{lemma}[theorem]{Lemma}

\newtheorem{proposition}[theorem]{Proposition}

\theoremstyle{definition}
\newtheorem{definition}[theorem]{Definition}

\theoremstyle{remark}

\numberwithin{equation}{section}

\begin{document}

\title[Non-ideal cyclotomic families]{On the non-idealness of cyclotomic families of pairing-friendly elliptic curves}

\author{Min Sha}
\address{Institut de Mathematiques de Bordeaux, Universite Bordeaux 1
, 33405 Talence Cedex, France}
\email{shamin2010@gmail.com}
\thanks{The author was supported by the China Scholarship Council.}






\begin{abstract}
Let $k=2^mp^n$ for an odd prime $p$ and integers
$m\ge 0$ and $n\ge 0$. We obtain lower bounds for the $\rho$-values of cyclotomic families of
pairing-friendly elliptic curves with embedding degree $k$ and $r(x)=\Phi_k(x)$.
Our bounds imply that none of these families are ideal.
\end{abstract}

\maketitle



\section{Introduction}
In recent years, the Tate pairing and the Weil pairing on elliptic curves over finite fields have been used to construct
 many novel cryptographic systems for which no other practical implementation is known; see \cite{Boneh2003,Boneh2004,Joux2000,Sakai2000} for the pioneering work and see \cite {Paterson2005} for a survey. The elliptic curves suitable for implementing pairing-based cryptographic systems are called \emph{pairing-friendly elliptic curves}.

More precisely, a pairing-friendly elliptic curve $E$ over a finite field $\F_q$
contains a subgroup of large prime-order $r$ such that for some $k$, $r|q^{k}-1$ and $r\nmid q^{i}-1$ for $0<i<k$, and
the parameters $q,r$ and $k$ should be chosen such that the discrete logarithm problem is infeasible both in an order-$r$ subgroup of $E(\F_q)$ and in $\F_{q^{k}}^{*}$, and the arithmetic in $\F_{q^{k}}$ is feasible. Here, $k$ is called the \emph{embedding degree} of $E$ with respect to $r$, and the ratio $\frac{\log q}{\log r}$ is called the \emph{$\rho$-value} of $E$ with respect to $r$.

Roughly speaking, pairing-friendly elliptic curves should have small embedding degree with respect to a large
prime-order subgroup. But Balasubramanian and Koblitz \cite{Koblitz1998} showed that, in general, the embedding degree $k$ can be expected to be around $r$. This makes pairing-friendly elliptic curves rare; for example see \cite{Shparlinski2012}. Thus, specific constructions are needed; see \cite{Freeman2010} for an exhaustive survey .

The main known strategy to construct pairing-friendly elliptic curves is as follows. Fix $k\ge 1$ and square-free $D\ge 1$, and look for an integer $t$ and two primes $r$ and $q$ satisfying
\begin{equation} \label{condition}
r|q+1-t,\quad r|\Phi_{k}(q),\quad 4q=t^{2}+Dy^{2} \textrm{ for some $y$ (\textit{CM equation})},
\end{equation}
where $\Phi_{k}(x)$ is the $k$-th cyclotomic polynomial, and $D$ is the so-called \emph{CM discriminant}. Then, the \emph{CM method} (see \cite[Section 18.1]{Cohen2005}) can produce an elliptic curve $E$ over $\F_q$ with $|E(\F_q)|=q+1-t$.

A well-known construction for the so-called \emph{complete families} of pairing-friendly elliptic curves with $k$ and $D$ fixed is due to \cite{Barreto2002,Brezing2005,Miyaji2001,Scott2006}.
Briefly speaking, the idea is to parameterize $t,r,  q,y$ as polynomials and then choose $ t(x),r(x), q(x)$, and $y(x)$ satisfying (\ref{condition})
 and such that $r(x)$ is irreducible and $q(x)$ is a power of an irreducible polynomial $p(x)$. If moreover $r(x)$ and $p(x)$ satisfy some extra conditions which conjecturably guarantee $r(x)$ and $p(x)$ to take infinitely many prime values simultaneously, then we say that the triple $( t(x),r(x), q(x))$ parameterizes a \emph{complete family} of elliptic curves. For such a family, the $\rho$-value, denoted by $\rho(t,r,q)$, is
 $$
 \rho(t,r,q)=\frac{\deg q(x)}{\deg r(x)}.
 $$
 When furthermore $r(x)$ is chosen to be $\Phi_{n}(x)$ with $k|n$, this yields the most popular family called \emph{cyclotomic family}; see Section \ref{family} for more details.

For practical considerations, neither $k$ nor $\deg r(x)$ should be large.
  Following \cite{Freeman2010}, a \emph{practical complete family} $( t(x),r(x), q(x))$ means that $k\le 50$ and $\deg r(x)\le 40$.

 In general, curves with small $\rho$-values are desirable in order to speed up arithmetic on the elliptic curves. The ideal case is $\rho=1$. We call a complete family $( t(x),r(x), q(x))$ of elliptic curves an \emph{ideal family} if $\rho(t,r,q)=1$.

Okano \cite{Okano2012} showed that for a cyclotomic family $( t(x),r(x), q(x))$ with embedding degree $k$ and $r(x)=\Phi_{k}(x)$, where $k=p$ or $2p$ for some odd prime (in this case, $p$ should be equal to 3 modulo 4), one has  $\rho(t,r,q)\ne 1$. By using some methods different from those in \cite{Okano2012} and based on the properties of cyclotomic polynomials, we can extend this result to more cases. In fact, we get much stronger results, that is, we obtain lower bounds or smallest possible values of $\rho(t,r,q)$.

\begin{theorem}\label{Thm1}
Let $k=2^{m}p^{n}$ for some odd prime $p$, and integers, $m\ge 0$ and $n\ge 0$, and $D$ a square-free positive integer. Suppose that $(t(x),r(x),q(x))$ is a cyclotomic family of elliptic curves with $r(x)=\Phi_{k}(x)$,
embedding degree $k$, and CM discriminant $D$. For the $\rho$-value $\rho(t,r,q)$, we list all the cases as follows.
\begin{enumerate}[$(1)$]
\item If $k=2^m$ (in this case we must have $m\ge 3$), we have $D=1$ or $D=2$. Then, the following hold. \label{2^m0}
   \begin{enumerate}
    \item[$(a)$] If $D=1$, then the smallest possible value of $\rho(t,r,q)$ is $1+\frac{1}{2^{m-2}}$.
    \item[$(b)$] If $D=2$, we have $\rho(t,r,q)\ge 1.5$.
   \end{enumerate}

\item If $k=p^n$ or $2p^n, n\ge 1$, we must have $p\equiv 3 \pmod 4$ and $D=p$. Then, the following hold. \label{p^n0}
   \begin{enumerate}
    \item[$(a)$] If $p=3$ (in this case we must have $n\ge 2$), then the smallest possible value of $\rho(t,r,q)$ is $1+\frac{1}{3^{n-1}}$.
    \item[$(b)$] If $p\ge 7$ , we have $\rho(t,r,q)> 1.1$.
   \end{enumerate}

\item If $k=2^mp^n, m\ge 2, n\ge 1$, we have $D=1,2,p$ or $2p$. Then, the following hold. \label{2^mp^n0}
   \begin{enumerate}
    \item[$(a)$] If $D=1$, then the smallest possible value of $\rho(t,r,q)$ is $\frac{p}{p-1}$.
    \item[$(b)$] If $D=2$ (in this case we must have $m\ge 3$), then $\rho(t,r,q)> 1.5$ for $p>3$, and $\rho(t,r,q)\ge 1.25$ for $p=3$.
    \item[$(c)$] If $D=p$ and $p\equiv 1 \pmod 4$, we have $\rho(t,r,q)> 1.4$ for $p>5$, and $\rho(t,r,q)\ge  1.75$ for $p=5$. If $D=p$, $p\equiv 3 \pmod 4$ and $p\ge 7$, we have $\rho(t,r,q)> 1.1$. If $D=p$ and $p=3$, then the smallest possible value of $\rho(t,r,q)$ is $1+\frac{1}{2^{m-1}3^{n-1}}$.
    \item[$(d)$] If $D=2p$ (in this case we must have $m\ge 3$), we have $\rho(t,r,q)> 1.5$.
   \end{enumerate}
\end{enumerate}
\end{theorem}

From Theorem \ref{Thm1}, we can see that for these cyclotomic families, small $\rho$-values can possibly occur only when $k=2^m, D=1$; or $k=3^n, 2\cdot 3^n, D=3$; or $k=2^mp^n, D=1$; or $k=2^m3^n, D=3$. This suggests that small $\rho$-values may occur rare for cyclotomic families with large CM discriminant $D$.

 In Theorem \ref{Thm1},
we say ``\emph{smallest possible value}'', because on the one hand the claimed $\rho$-value is indeed a lower bound for the $\rho$-values of such cyclotomic families, on the other hand in Section \ref{Proof1} we indeed get a polynomial triple $(t(x),r(x),q(x))$ with the claimed $\rho$-value following Theorem \ref{Brezing-Weng}, but one needs to check whether it is a complete family. Although it is not easy to get a general result,
with the help of the computer algebra system PARI/GP \cite{Pari},
for $k\le 82$ we get Table \ref{data} (see Section \ref{Proof2}),
which is compatible with Theorem \ref{Thm1}.

The computations for Table \ref{data} suggest that some smallest possible $\rho$-values can be achieved,
 for example Theorem \ref{Thm1} ($2a$) when $k=3^n$ and ($3c$) when $k=3\cdot 2^m, m\ge 2$; but some cannot be achieved, for example Theorem \ref{Thm1} ($1a$) and ($3a$). The reader can also check this in Table \ref{data}.

Comparing Table \ref{data} with \cite[Table 5]{Freeman2010},
the cyclotomic families in Theorem \ref{Thm1} ($2a$) when $k=9$ or $k=27$ may be comparable,
since they have the same $\rho$-values as claimed in \cite[Table 5]{Freeman2010}
and they have a simpler form of $t(x)$. We list them explicitly as follows.

$k=9,D=3:$
\begin{equation}
\left\{ \begin{array}{ll}
                t(x)=x+1  ,\\
                r(x)=x^6+x^3+1  ,\\
                q(x)=\frac{1}{3}(x-1)^2(x^6+x^3+1)+x.
                 \end{array} \right.
\notag
\end{equation}

$k=27,D=3:$
\begin{equation}
\left\{ \begin{array}{ll}
                t(x)=x+1  ,\\
                r(x)=x^{18}+x^9+1  ,\\
                q(x)=\frac{1}{3}(x-1)^2(x^{18}+x^9+1)+x.
                 \end{array} \right.
\notag
\end{equation}

The following theorem is a direct corollary of Theorem \ref{Thm1}, which says that all the cyclotomic families in Theorem \ref{Thm1} are not ideal.

\begin{theorem}\label{Thm2}
Let $k=2^{m}p^{n}$ for some odd prime $p$, and integers, $m\ge 0$ and $n\ge 0$, and $D$ a square-free positive integer. Suppose that $(t(x),r(x),q(x))$ is a cyclotomic family of elliptic curves with $r(x)=\Phi_{k}(x)$,
embedding degree $k$, and CM discriminant $D$.
Then, we have
$$
\rho(t,r,q)\ne 1.
$$
\end{theorem}

 Although Table 5 in \cite{Freeman2010} suggests that there are no practical ideal cyclotomic families of pairing-friendly elliptic curves, there are few general results. Theorem \ref{Thm2} gives evidence that there is no ideal cyclotomic family of pairing-friendly elliptic curves.

 We want to indicate that the methods we use here rely heavily on simple expressions of the cyclotomic polynomials $\Phi_{2^{m}p^{n}}(x)$. For a general cyclotomic family $(t(x),r(x), q(x))$ with embedding degree $k$, $r(x)=\Phi_{kd}(x)$ and $d\ge 2$, it might be difficult to
 get a general result like Theorem \ref{Thm1} or Theorem \ref{Thm2}.

\section{Complete families of pairing-friendly elliptic curves}\label{family}

In this section, we will briefly introduce complete families of pairing-friendly elliptic curves; 
see \cite[Section 6]{Freeman2010} for more details.

A famous conjecture of Buniakowski and Schinzel \cite[Page 323]{Lang} asserts that a non-constant $f(x)\in \Z[x]$
takes an infinite number of prime values if and only if $f(x)$ is irreducible with positive leading coefficient, and $\gcd(\{f(x) : x\in\Z\}) = 1$. Furthermore, a conjecture by Bateman and Horn \cite{Bateman1962} predicts the density of such prime values. In practice, we must also consider rational polynomials.

\begin{definition}
We say that a polynomial $f(x)\in \Q[x]$ \emph{represents integers} if $f(x)\in \Z$ for some $x\in \Z$.
\end{definition}

\begin{definition}\label{reprime}
We say that $f(x)\in \Q[x]$ \emph{represents primes} if it satisfies the following conditions:
\begin{enumerate}[$(1)$]
\item $f(x)$ is non-constant and irreducible with positive leading coefficient;
\item $f(x)$ represents integers;
\item $\gcd(\{f(x)\in\Z:x\in\Z\})=1$.
\end{enumerate}
\end{definition}

So, when a rational polynomial $f(x)$ represents primes, it is likely to take infinitely many prime values. Now we are ready to define complete families of elliptic curves.
\begin{definition}
For a given positive integer $k$ and a positive square-free integer $D$, the triple
$(t(x), r(x), q(x))\in \Q[x]^{3}$ \emph{parameterizes a complete family of elliptic curves with embedding degree $k$ and CM discriminant $D$} if the following conditions are satisfied:
\begin{enumerate}[$(1)$]
\item $q(x)$ is a power of a polynomial which represents primes;
\item $r(x)$ represents primes and $t(x)$ represents integers;
\item $r(x)|q(x)+1-t(x)$ and $r(x)|\Phi_{k}(t(x)-1)$;
\item There exists some $y(x)\in\Q[x]$ representing integers such that $4q(x)=t(x)^2+Dy(x)^2$.
\end{enumerate}
\end{definition}

Barreto, Lynn and Scott \cite{Barreto2002} and (independently) Brezing and Weng \cite{Brezing2005} both observed that we can generalize the Cocks-Pinch method (see \cite[Theorem 4.1]{Freeman2010}) to produce complete families of elliptic curves. Brezing and Weng gave a construction in greatest generality. We describe it below as stated in \cite{Freeman2010} with minor modifications.
\begin{theorem}[Brezing-Weng \cite{Brezing2005}]\label{Brezing-Weng}
Fix a positive integer $k$ and a positive square-free integer $D$. Then execute the following steps.
\begin{enumerate}[$(1)$]
\item Find an irreducible polynomial $r(x)\in\Z[x]$ with positive leading coefficient such that a number field $K\cong\Q[x]/(r(x))$  contains $\sqrt{-D}$ and the $k$-th cyclotomic field.
\item Choose a primitive $k$-th root of unity $\eta_k\in K$.
\item Let $t(x)\in\Q[x]$ be a polynomial mapping to $\eta_k+1$ in K such that $\deg t(x)<\deg r(x)$.
\item Let $y(x)\in\Q[x]$ be a polynomial mapping to $(\eta_k-1)/\sqrt{-D}$ in K such that $\deg y(x)<\deg r(x)$.
\item Let $q(x)\in\Q[x]$ be given by $(t(x)^2+Dy(x)^2)/4$.
\end{enumerate}
Suppose that $q(x)$ represents primes and both $t(x)$ and $y(x)$ represent integers. Then the triple
$(t(x), r(x), q(x))$ parameterizes a complete family of elliptic curves with embedding degree $k$ and CM discriminant $D$. The $\rho$-value of this family is
$$
\rho(t,r,q)=\frac{2\max\{\deg t(x),\deg y(x)\}}{\deg r(x)}.
$$
\end{theorem}

The \emph{cyclotomic families} of elliptic curves are exactly constructed
by the Brezing-Weng method when $r(x)$ is taken to be a cyclotomic polynomial
$\Phi_{n}(x)$ in the above theorem (in this case, $k|n$).

Searching for ideal complete families of elliptic curves is still an important open problem in pairing-based cryptography. So far there is only one known ideal complete family, constructed by Barreto and Naehrig \cite{Barreto2006} with $k=12$ and $D=3$. We state it as follows:
\begin{equation}
\left\{ \begin{array}{ll}
                t(x)=6x^2+1  ,\\
                r(x)=36x^4+36x^3+18x^2+6x+1  ,\\
                q(x)=36x^4+36x^3+24x^2+6x+1.
                 \end{array} \right.
\notag
\end{equation}

\section{Proof of Theorem \ref{Thm1}} \label{Proof1}
In this section, we will prove Theorem \ref{Thm1} by cases.

First, under the assumptions of Theorem \ref{Thm1} and then using Theorem \ref{Brezing-Weng},
we fix some notation.

Put $L=\Q[x]/(r(x))$, where $r(x)=\Phi_{k}(x)$. Let $\zeta_k=\exp(2\pi \sqrt{-1}/k)$.
Notice that there is a canonical isomorphism between $L$ and $\Q(\zeta_k)$,
and $x$ is a primitive $k$-th root of unity in $L$.  From now on, we fix the isomorphism
$$
\sigma: L \to \Q(\zeta_k),\quad x \mapsto \zeta_k.
$$

Then, $t(x)$ maps to $\zeta_{k}^g+1$ for some integer $g$ such that $1\le g \le k$ and $\gcd(g,k)=1$, so
$$
t(x)\equiv x^g+1 \quad \textrm{(mod $r(x)$)}.
$$
Thus, $y(x)$ maps to $-\frac{1}{D}(\zeta_{k}^g-1)\sqrt{-D}$.
Let $y_{1}(x)$ map to $(\zeta_{k}^g-1)\sqrt{-D}$ with degree less than $\deg r(x)$.
If $s(x)$ maps to $\sqrt{-D}$, we have
$$
y_{1}(x)\equiv (x^{g}-1)s(x) \quad \textrm{(mod $r(x)$) with $\deg y_{1}(x)<\deg r(x)$}.
$$

Obviously, we have $y(x)=-\frac{1}{D} y_{1}(x)$ and $\deg y(x)=\deg y_{1}(x)$.
Then, for the $\rho$-value, we have
$$
\rho(t,r,q)=\frac{2\max\{\deg t(x), \deg y_1(x)\}}{\varphi(k)},
$$
where $\varphi$ is Euler's totient function. Thus, for bounding $\rho(t,r,q)$, we need to estimate $\deg t(x)$ or $\deg y_1(x)$.

Here, we want to indicate a simple fact which will be used several times later on.
Let $\mathfrak{f}$ be the conductor of the quadratic field $\Q(\sqrt{-D})$,
and in fact $\mathfrak{f}$ is equal to the absolute value of its discriminant.
By Kronecker-Weber theorem, $\mathfrak{f}$ is the smallest integer $n$
such that the $n$-th cyclotomic field contains $\sqrt{-D}$.
Since $\sqrt{-D}\in \Q(\zeta_k)$ and using the discriminant formula of quadratic fields, we have
\begin{equation}\label{Dk}
\left\{ \begin{array}{ll}
                 D|k  & \textrm{if $D\equiv 3$ (mod 4)},\\
                 4D|k  & \textrm{otherwise}.
                 \end{array} \right.
\end{equation}

\begin{proposition}\label{2^m}
Let $k=2^m$ for some integer $m\ge 0$, then Theorem $\ref{Thm1}$ $(\ref{2^m0})$ is true.
\end{proposition}
\begin{proof}
If $k=1,2$, by \eqref{Dk}, the cyclotomic family $(t(x),r(x),q(x))$ actually does not exist.
So, we must have $m\ge 2$. Later on, we will see that actually we must have $m\ge 3$.

By \eqref{Dk}, if $m=2$, then we must have $D=1$; otherwise if $m\ge 3$, then $D=1,$ or 2.

First, assume that $D=1$. Notice that $\sqrt{-1}=\zeta_{k}^{2^{m-2}}$.
Then $y_{1}(x)\equiv (x^{g}-1)x^{2^{m-2}}\equiv x^{2^{m-2}+g}-x^{2^{m-2}}$ (mod $r(x)$).
Here, $r(x)=x^{2^{m-1}}+1$.
For $m=2$, by direct calculations following Theorem \ref{Brezing-Weng}, we have
$q(x)=\frac{1}{2}(x+1)^2$ if $g=1$, and $q(x)=\frac{1}{2}(x-1)^2$ if $g=3$, both of them do not represent primes,
so the cyclotomic family $(t(x),r(x),q(x))$ actually does not exist. Thus, we must have $m\ge 3$.
Notice that $g$ is an odd integer. It is straightforward to show that
\begin{equation*}
\left\{ \begin{array}{ll}
                 \deg t(x)=g, \, \deg y_{1}(x)=2^{m-2}+g  & \textrm{if $1< g < 2^{m-2}$},\\
                 \deg t(x)=g, \, \deg y_{1}(x)=2^{m-2}  & \textrm{if $2^{m-2}< g< 2^{m-1}$},\\
                 \deg t(x)=g-2^{m-1}, \, \deg  y_{1}(x)=g-2^{m-2} & \textrm{if $2^{m-1}< g< 2^{m-1}+2^{m-2}$},\\
                 \deg t(x)=g-2^{m-1}, \, \deg y_{1}(x)=2^{m-2}  & \textrm{if $2^{m-1}+2^{m-2}< g<k$}.
                 \end{array} \right.
\end{equation*}
Thus, we always have $\max\{\deg t(x),\deg y_{1}(x)\}\ge 2^{m-2}+1$,
which implies that
\begin{equation}
\rho(t,r,q)\ge \frac{2(2^{m-2}+1)}{2^{m-1}}=1+\frac{1}{2^{m-2}},
\end{equation}
where the equality can possibly be achieved when $g=1$, that is $t(x)=x+1$.

Now assume that $D=2$. In this case, we must have $m\ge 3$. It is easy to see that $\sqrt{-2}=\zeta_{8}^{3}+\zeta_{8}=\zeta_{k}^{3\cdot 2^{m-3}}+\zeta_{k}^{2^{m-3}}$.
Then we have $y_{1}(x)\equiv x^{2^{m-2}+2^{m-3}+g}+x^{2^{m-3}+g}-x^{2^{m-2}+2^{m-3}}-x^{2^{m-3}}$ (mod $r(x)$).
Notice that $g$ is an odd integer and $2^{m-2}+2^{m-3}<2^{m-1}$,
then in view of the form of $r(x)$, the term $-x^{2^{m-2}+2^{m-3}}$ is always non-vanishing.
So, $\deg y_{1}(x)\ge 2^{m-2}+2^{m-3}$, which implies that
\begin{equation}\label{2^m2}
\rho(t,r,q)\ge 2(2^{m-2}+2^{m-3})/2^{m-1}=1.5.
\end{equation}

Therefore, we complete the proof of Proposition \ref{2^m}.
\end{proof}

\begin{proposition}\label{p^n}
For any odd prime $p$ and integer $n\ge 1$, Theorem $\ref{Thm1}$ $(\ref{p^n0})$ is true for $k=p^n$.
\end{proposition}
\begin{proof}
If $p\equiv 1$ (mod 4), by \eqref{Dk}, we only possibly have $D=1$ or $k=p$, both of which lead to $4|k$. 
This is impossible, so the cyclotomic family $(t(x),r(x),q(x))$ actually does not exist.
So, we must have $p\equiv 3$ (mod 4).

Since $k=p^n$ with $p\equiv 3$ (mod 4) and applying \eqref{Dk}, we must have $D=p$. It is well-known that
$$
\sqrt{-p}=\sum\limits_{a=1}^{p-1}\left(\frac{a}{p}\right)\zeta_{p}^{a}=
\sum\limits_{a=1}^{p-1}\left(\frac{a}{p}\right)\zeta_{k}^{ap^{n-1}},
$$
where $(\frac{\cdot}{\cdot})$ is the Legendre symbol, see \cite[Theorem 7 on p. 349]{Shafarevich1966}.
Thus,
\begin{equation}\label{p^ny_1}
y_{1}(x)\equiv \sum\limits_{a=1}^{p-1}\left(\frac{a}{p}\right)x^{ap^{n-1}+g}-
\sum\limits_{a=1}^{p-1}\left(\frac{a}{p}\right)x^{ap^{n-1}} \quad
\textrm{(mod $r(x)$)}.
\end{equation}

We first suppose that $p=3$. Then we have
$$
y_{1}(x)\equiv 2x^{3^{n-1}+g}+x^{g}-2x^{3^{n-1}}-1 \quad \textrm{(mod $r(x)$)},
$$
where $r(x)=x^{2\cdot 3^{n-1}}+x^{3^{n-1}}+1$.
For $n=1$, that is $k=3$, by direct calculations following Theorem \ref{Brezing-Weng}, we have
$q(x)=(x+1)^2$ if $g=1$, and $q(x)=x^2$ if $g=2$, both of them do not represent primes,
so the cyclotomic family $(t(x),r(x),q(x))$ actually does not exist.
So, when $p=3$, we must have $n\ge 2$.
Notice that $\gcd(g,3)=1$ and $n\ge 2$. It is straightforward to show that
\begin{equation}\label{3^n3}
\left\{ \begin{array}{ll}
                 \deg t(x)=g, \, \deg y_{1}(x)=3^{n-1}+g  & \textrm{if $1\le g<3^{n-1}$},\\
                 \deg t(x)=g, \, \deg y_{1}(x)=g  & \textrm{if $3^{n-1}< g<2\cdot 3^{n-1}$},\\
                 \deg t(x)=g-3^{n-1}, \, \deg y_{1}(x)=g-3^{n-1} & \textrm{if $2\cdot 3^{n-1}< g<k$}.
                 \end{array} \right.
\notag
\end{equation}
Thus, we always have $\max\{\deg t(x),\deg y_{1}(x)\}\ge 3^{n-1}+1$, which implies that
\begin{equation}
\rho(t,r,q)\ge \frac{2(3^{n-1}+1)}{2\cdot 3^{n-1}}=1+\frac{1}{3^{n-1}},
\end{equation}
where the equality can possibly be achieved when $g=1$, that is $t(x)=x+1$.

Now we assume that $p>3$ and $n\ge 2$. By the choice of $p$, we have $p\ge 7$. Put
$$
y_{2}(x)\equiv \sum\limits_{a=1}^{p-1}\left(\frac{a}{p}\right)x^{ap^{n-1}}
\quad \textrm{(mod $r(x)$)\quad with $\deg y_{2}(x)<\deg r(x)$}.
$$
Since $r(x)=\sum\limits_{a=0}^{p-1}x^{ap^{n-1}}$ and $n\ge 2$,
every exponent of $x$ appearing in $r(x)$ is divisible by $p$.
Notice that every $ap^{n-1}+g$ is coprime to $p$.
So $$\deg y_{1}(x)\ge \deg y_{2}(x).$$
 Then it suffices to consider $y_{2}(x)$.

Since $\left(\frac{p-1}{p}\right)=-1$, it is easy to see that
\begin{equation}
y_{2}(x)\equiv 1+\sum\limits_{a=1}^{p-2}(1+\left(\frac{a}{p}\right))x^{ap^{n-1}}
\quad \textrm{(mod $r(x)$)}.
\notag
\end{equation}
Note that the degree of the above polynomial on the right hand side is less than $\deg r(x)$,
we have $y_{2}(x)= 1+\sum\limits_{a=1}^{p-2}(1+\left(\frac{a}{p}\right))x^{ap^{n-1}}$.
Note that, for $p\ge 7$, there exists an integer $b$ such that $\sqrt{11p/20}<b<\sqrt{p}$.
Indeed, when $p\ge 17$, we have $\sqrt{p}-\sqrt{11p/20}>1$; for other $p\ge 7$,
one can verify it by direct computation. Combining with $\left(\frac{p-1}{p}\right)=-1$, we have
\begin{equation}\label{Legendre}
\left(\frac{b^{2}}{p}\right)=1 \quad \textrm{and} \quad 11p/20<b^2<p-1.
\end{equation}
Thus, $\deg y_2(x)\ge b^2p^{n-1}>11p^n/20$, which implies that
\begin{equation}\label{p^n7}
\rho(t,r,q)>(2\cdot 11p^n/20)/(p^{n-1}(p-1))>1.1.
\end{equation}

Finally, we assume that $p>3$ and $n=1$. Okano \cite{Okano2012} showed that $\rho(t,r,q)\ne 1$.
Indeed, by combining \eqref{Legendre} and the later part in \cite[Proof of Proposition 4.2]{Okano2012},
one can similarly deduce that $\rho(t,r,q)>1.1$.

Therefore, Proposition \ref{p^n} has been proved.
\end{proof}

\begin{proposition}\label{2p^n}
For any odd prime $p$ and integer $n\ge 1$,  Theorem $\ref{Thm1}$ $(\ref{p^n0})$ is true for $k=2p^n$.
\end{proposition}
\begin{proof}
If $p\equiv 1$ (mod 4), similarly as before, the cyclotomic family $(t(x),r(x),q(x))$ does not exist.
So, we must have $p\equiv 3$ (mod 4).

Since $k=2p^n$ with $p\equiv 3$ (mod 4), by \eqref{Dk}, we must have $D=p$.
Notice that $\zeta_k=-\zeta_{p^n}$, we have
$$
\sqrt{-p}=\sum\limits_{a=1}^{p-1}\left(\frac{a}{p}\right)\zeta_{p^n}^{ap^{n-1}}=
\sum\limits_{a=1}^{p-1}\left(\frac{a}{p}\right)(-\zeta_{k})^{ap^{n-1}}.
$$
Then
$$
y_{1}(x)\equiv \sum\limits_{a=1}^{p-1}\left(\frac{a}{p}\right)(-x)^{ap^{n-1}+g}-
\sum\limits_{a=1}^{p-1}\left(\frac{a}{p}\right)(-x)^{ap^{n-1}} \quad
\textrm{(mod $r(x)$)}.
$$
Notice that $r(x)=\sum\limits_{a=0}^{p-1}(-x)^{ap^{n-1}}$. If we put $z=-x$,
in view of \eqref{p^ny_1}, we almost reduce our proof to the case of Proposition \ref{p^n}
except that $t(x)$ is congruent to $-z^g+1$ modulo $\sum\limits_{a=0}^{p-1}z^{ap^{n-1}}$ after substitution.
Thus, Proposition \ref{2p^n} can be proved similarly when $p=3, n\ge 1$, or $p> 3, n\ge 2$.

Now, we assume that $p>3$ and $n=1$. Okano \cite{Okano2012} showed that $\rho(t,r,q)\ne 1$.
Indeed, by combining \eqref{Legendre} and the later part in \cite[Proof of Proposition 4.2]{Okano2012},
one can also deduce that $\rho(t,r,q)>1.1$.

Therefore, we complete the proof of Proposition \ref{2p^n}.
\end{proof}

To handle more complicated cases, we need some preparations.

For an odd prime $p$, integers $m\ge 1$ and $n\ge 1$, it is well-known that
$$
\Phi_{2^mp^n}(x)=\sum\limits_{a=0}^{p-1}(-1)^{a}x^{a2^{m-1}p^{n-1}}.
$$
 For each integer $i\ge 1$, define
$$
f_{i}(x)\equiv x^{i} \quad \textrm{(mod $\Phi_{2^mp^n}(x)$) with $\deg f_{i}(x)<\deg \Phi_{2^mp^n}(x)$}.
$$
By definition, every $f_{i}(x)$ is well-defined and unique.
\begin{lemma}
For integer $i$ with $1\le i<2^{m-1}p^{n-1}(p-1)$, we have $f_{i}(x)=x^i$ .
\end{lemma}

\begin{lemma}\label{fn}
For any odd prime $p$ and positive integers $j,m,n$, let $i=j\cdot 2^{m-1}p^{n}$. Then, we have
\begin{equation}
\left\{ \begin{array}{ll}
                 f_{i}(x)=(-1)^j,\\
                 f_{i-h}(x)=\sum\limits_{a=0}^{p-2}(-1)^{a+j}x^{(a+1)2^{m-1}p^{n-1}-h} & \textrm{if $1\le h\le 2^{m-1}p^{n-1}$},\\
                 f_{i+h}(x)=(-1)^{j}x^h  & \textrm{if $1\le h< 2^{m-1}p^{n-1}(p-1)$}.
                 \end{array} \right.
\notag
\end{equation}
\end{lemma}
\begin{proof}
We prove this lemma by induction. For $j=1$, it is straightforward to verify the desired formulas. Now for $j\ge 2$, assume that the desired formulas are true for $j-1$.

Put $b=j\cdot 2^{m-1}p^{n}-2^{m-1}p^{n-1}$. Then all we need to do is to compute $f_{b}(x)$. Notice that $b-1=(j-1)2^{m-1}p^{n}+2^{m-1}p^{n-1}(p-1)-1$, by the assumption we have $f_{b-1}(x)=(-1)^{j-1}x^{2^{m-1}p^{n-1}(p-1)-1}$. Thus
\begin{align*}
f_{b}(x)&\equiv (-1)^{j-1}x^{2^{m-1}p^{n-1}(p-1)}\\
&\equiv \sum\limits_{a=0}^{p-2}(-1)^{a+j}x^{a2^{m-1}p^{n-1}} \quad \textrm{(mod $\Phi_{2^mp^n}(x)$)}.
\end{align*}
Then the other formulas follow easily.
\end{proof}

\begin{proposition}\label{2^mp^n3}
For any odd prime $p$ with $p\equiv 3$ {\rm(mod 4)} and integers, $m\ge 2$ and $n\ge 1$,
 Theorem $\ref{Thm1}$ $(\ref{2^mp^n0})$ is true for $k=2^mp^n$.
\end{proposition}
\begin{proof}
By \eqref{Dk}, if $m=2$, then $D=1$ or $D=p$; otherwise if $m\ge 3$, then $D=1,2,p$, or $2p$.

First, assume that $D=1$. Since $\sqrt{-1}=\zeta_{k}^{2^{m-2}p^n}$, we have
$$
y_{1}(x)\equiv x^{2^{m-2}p^n}(x^g-1)\equiv x^{2^{m-2}p^n+g}-x^{2^{m-2}p^n} \quad \textrm{(mod $r(x)$)}.
$$
Here, $r(x)=\sum\limits_{a=0}^{p-1}(-1)^{a}x^{a2^{m-1}p^{n-1}}$,
which is also a polynomial with respect to $x^{2}$.
Since $g$ is an odd integer, the two integers $2^{m-2}p^n+g$ and $2^{m-2}p^n$ have different parities.
Note that $2^{m-2}p^n< \deg r(x)$. So, the term $-x^{2^{m-2}p^n}$ does not vanish,
and then we have $\deg y_{1}(x)\ge 2^{m-2}p^n$, which implies that
\begin{equation}
\rho(t,r,q)\ge \frac{2\cdot 2^{m-2}p^n}{2^{m-1}p^{n-1}(p-1)}=\frac{p}{p-1}.
\end{equation}
We claim that the above equality can possibly be achieved. Indeed, if $m=2$, we choose $g=p^n+2<k$,
then $\gcd(g,k)=1$, and by Lemma \ref{fn}, we have $y_1(x)=-x^2-x^{p^n}$, so we get the equality.
Otherwise if $m>2$, we choose $g=2^{m-2}p^n+1<k$,
then $\gcd(g,k)=1$, and by Lemma \ref{fn}, we have $y_1(x)=-x-x^{2^{m-2}p^n}$,
so we still get the equality.

Suppose that $D=2$. Note that in this case we must have $m\ge 3$. Since $\sqrt{-2}=\zeta_{8}^{3}+\zeta_{8}=\zeta_{k}^{3\cdot 2^{m-3}p^n}+\zeta_{k}^{2^{m-3}p^n}$, we have
$$
y_{1}(x)\equiv x^{3\cdot 2^{m-3}p^n+g}+x^{2^{m-3}p^n+g}-x^{3\cdot 2^{m-3}p^n}-x^{2^{m-3}p^n} \quad
\textrm{(mod $r(x)$)}.
$$
Notice that the two integers $3\cdot 2^{m-3}p^n+g$ and $2^{m-3}p^n+g$ have the same parity,
and the two integers $3\cdot 2^{m-3}p^n$ and $2^{m-3}p^n$ have the same parity,
but $2^{m-3}p^n+g$ and $2^{m-3}p^n$ have different parities. For $p>3$,
note that $3\cdot 2^{m-3}p^n<\deg r(x)$, then the term $-x^{3\cdot 2^{m-3}p^n}$ does not vanish,
so we have $\deg y_{1}(x)\ge 3\cdot 2^{m-3}p^n$, which implies that
\begin{equation}\label{2^mp^n27}
\rho(t,r,q)\ge 3\cdot 2^{m-2}p^n/(2^{m-1}p^{n-1}(p-1))>1.5.
\end{equation}

If $D=2$ and $p=3$, we have
$$
y_{1}(x)\equiv x^{2^{m-3}3^{n+1}+g}+x^{2^{m-3}3^n+g}-x^{2^{m-3}3^{n+1}}-x^{2^{m-3}3^n} \quad
\textrm{(mod $r(x)$)},
$$
where $r(x)=x^{2^{m}3^{n-1}}-x^{2^{m-1}3^{n-1}}+1$. Then we obtain
$$
y_{1}(x)\equiv x^{2^{m-3}3^{n+1}+g}+x^{2^{m-3}3^n+g}-x^{5\cdot 2^{m-3}3^{n-1}}+x^{2^{m-3}3^{n-1}}-x^{2^{m-3}3^n} \quad
\textrm{(mod $r(x)$)}.
$$
Notice that $\deg r(x)>5\cdot 2^{m-3}3^{n-1}$. As before, we have $\deg y_{1}(x)\ge 5\cdot 2^{m-3}3^{n-1}$,
which implies that
\begin{equation}\label{2^mp^n23}
\rho(t,r,q)\ge 10\cdot 2^{m-3}3^{n-1}/(2^{m}3^{n-1})=1.25.
\end{equation}

Now suppose that $D=p$. Since $\sqrt{-p}=\sum\limits_{a=1}^{p-1}\left(\frac{a}{p}\right)\zeta_{k}^{a2^{m}p^{n-1}}$, we have
$$
y_{1}(x)\equiv \sum\limits_{a=1}^{p-1}\left(\frac{a}{p}\right)x^{a2^{m}p^{n-1}+g}-
\sum\limits_{a=1}^{p-1}\left(\frac{a}{p}\right)x^{a2^{m}p^{n-1}} \quad
\textrm{(mod $r(x)$)}.
$$
Put
$$
y_{2}(x)\equiv \sum\limits_{a=1}^{p-1}\left(\frac{a}{p}\right)x^{a2^{m}p^{n-1}} \quad \textrm{(mod $r(x)$)\quad with $\deg y_{2}(x)<\deg r(x)$}.
$$
Notice that $r(x)$ is also a polynomial with respect to $x^2$,
and every $a2^{m}p^{n-1}+g$ is an odd integer.
So, the leading term of $-y_2(x)$ appears in $y_1(x)$.
Thus, we have
$$\deg y_{1}(x)\ge \deg y_{2}(x).$$
Then we only need to consider $y_{2}(x)$. Put $z=x^{2^{m-1}p^{n-1}}$, then $r(x)=\Phi_{2p}(z)$.
Define
$$
y_{3}(z)\equiv \sum\limits_{a=1}^{p-1}\left(\frac{a}{p}\right)z^{2a} \quad \textrm{(mod $\Phi_{2p}(z)$)\quad with $\deg y_{3}(z)<\deg \Phi_{2p}(z)$}.
$$
Then $\deg y_{2}(x)=2^{m-1}p^{n-1}\deg y_{3}(z)$.

Applying Lemma \ref{fn} by setting $m=1$ and $n=1$, and noticing that $\left(\frac{a}{p}\right)=\left(\frac{4a}{p}\right)$
for any integer $a$, we get
\begin{align*}
y_{3}(z)
&= \sum\limits_{a=1}^{(p-3)/2}\left(\frac{a}{p}\right)z^{2a}+
\left(\frac{2}{p}\right)\sum\limits_{a=0}^{p-2}(-z)^{a}-
\sum\limits_{a=(p+1)/2}^{p-1}\left(\frac{a}{p}\right)z^{2a-p}\\
&=\sum\limits_{\substack{a=2\\\textrm{$a$ even}}}^{p-3}\left(\frac{2a}{p}\right)z^{a}+\left(\frac{2}{p}\right)\sum\limits_{a=0}^{p-2}(-z)^{a}-
\sum\limits_{\substack{a=1\\ \textrm{$a$ odd}}}^{p-2}\left(\frac{2a}{p}\right)z^{a}\\
&=\left(\frac{2}{p}\right)\sum\limits_{a=0}^{p-2}(1+\left(\frac{a}{p}\right))(-z)^{a}.
\end{align*}

As \eqref{Legendre}, for $p\ge 7$, there exists an integer $b$ such that
\begin{equation}
\left(\frac{b^{2}}{p}\right)=1 \quad \textrm{and} \quad 11p/20<b^2<p-1.
\notag
\end{equation}
 So, $\deg y_{3}(z)\ge b^2$, which implies that for $p\ge 7$, we have
\begin{equation}\label{2^mp^n7}
\rho(t,r,q)\ge 2b^2/(p-1)>(2\cdot 11p/20)/(p-1)>1.1.
\end{equation}

For the case $D=p$ and $p=3$, we have
$$
y_{1}(x)\equiv 2x^{2^{m-1}3^{n-1}+g}-x^{g}-2x^{2^{m-1}3^{n-1}}+1 \quad \textrm{(mod $r(x)$)},
$$
where $r(x)=x^{2^{m}3^{n-1}}-x^{2^{m-1}3^{n-1}}+1$.
Let $w_1(x)\equiv x^g$ (mod $r(x)$) with $\deg w_1(x)<\deg r(x)$,
and denote $d=\deg w_1(x)$. Since $g$ is an odd integer, $d$ is also odd.
If $d<2^{m-1}3^{n-1}$, then $d+2^{m-1}3^{n-1}<\deg r(x)$,
so in view of the term $2x^{2^{m-1}3^{n-1}+g}$, we have  $\deg y_1(x)=d+2^{m-1}3^{n-1}$.
If $d>2^{m-1}3^{n-1}$, let $w_2(x)\equiv 2x^{2^{m-1}3^{n-1}+g}$ (mod $r(x)$)
with $\deg w_2(x)<\deg r(x)$, by Lemma \ref{fn} 
the coefficient of $x^d$ in $w_1(x)$ is $\pm 1$,
but the coefficients in $w_2(x)$ are $\pm 2$, so the term $x^d$ does not vanish,
then $\deg y_1(x)\ge d$. Thus, we always have $\deg y_1(x)\ge 2^{m-1}3^{n-1}+1$, which implies that
\begin{equation}
\rho(t,r,q)\ge \frac{2(2^{m-1}3^{n-1}+1)}{2^m3^{n-1}}=1+\frac{1}{2^{m-1}3^{n-1}},
\end{equation}
where the equality can possibly be achieved when $g=1$, that is $t(x)=x+1$.

Now assume that $D=2p$. In this case, we must have $m\ge 3$. Since $\sqrt{2}=-\zeta_{8}^{3}+\zeta_8=-\zeta_{k}^{3\cdot 2^{m-3}p^n}+\zeta_{k}^{2^{m-3}p^n}$, we have
$$
\sqrt{-2p}=\sqrt{2}\cdot \sqrt{-p}=-\sum\limits_{a=1}^{p-1}\left(\frac{a}{p}\right)\zeta_{k}^{2^{m-3}p^{n-1}(8a+3p)}+
\sum\limits_{a=1}^{p-1}\left(\frac{a}{p}\right)\zeta_{k}^{2^{m-3}p^{n-1}(8a+p)}.
$$
Then
$$
y_{1}(x)\equiv (x^g-1)\left(-\sum\limits_{a=1}^{p-1}\left(\frac{a}{p}\right)x^{2^{m-3}p^{n-1}(8a+3p)}+
\sum\limits_{a=1}^{p-1}\left(\frac{a}{p}\right)x^{2^{m-3}p^{n-1}(8a+p)}\right) \quad
\textrm{(mod $r(x)$)}.
$$
Put
$$
y_{2}(x)\equiv \sum\limits_{a=1}^{p-1}\left(\frac{a}{p}\right)x^{2^{m-3}p^{n-1}(8a+3p)}-\sum\limits_{a=1}^{p-1}\left(\frac{a}{p}\right)x^{2^{m-3}p^{n-1}(8a+p)} \quad
\textrm{(mod $r(x)$)}
$$
with $\deg y_{2}(x)< \deg r(x)$. Similarly as before, we have
$$\deg y_{1}(x)\ge \deg y_{2}(x).$$
 Then we only need to consider $y_{2}(x)$. Set $z=x^{2^{m-3}p^{n-1}}$, then $r(x)=\Phi_{8p}(z)$. Define
$$
y_{3}(z)\equiv \sum\limits_{a=1}^{p-1}\left(\frac{a}{p}\right)z^{8a+3p}-\sum\limits_{a=1}^{p-1}\left(\frac{a}{p}\right)z^{8a+p} \,\, \textrm{(mod $\Phi_{8p}(z)$)\, with $\deg y_{3}(z)<\deg \Phi_{8p}(z)$}.
$$
Obviously, we have
$$\deg y_{2}(x)=2^{m-3}p^{n-1}\deg y_{3}(z).$$
 Define
\begin{align*}
&y_{31}(z)\equiv \sum\limits_{a=1}^{p-1}\left(\frac{a}{p}\right)z^{8a+3p} \quad \textrm{(mod $\Phi_{8p}(z)$)\quad with $\deg y_{31}(z)<\deg \Phi_{8p}(z)$},\\
&y_{32}(z)\equiv \sum\limits_{a=1}^{p-1}\left(\frac{a}{p}\right)z^{8a+p} \quad \textrm{(mod $\Phi_{8p}(z)$)\quad with $\deg y_{32}(z)<\deg \Phi_{8p}(z)$}.
\end{align*}
Then $y_{3}(z)=y_{31}(z)-y_{32}(z)$. Since $\Phi_{8p}(z)$ is a polynomial with respect to $z^{4}$ and for any two integers $a$ and $b$, $8a+3p$ is not congruent to $8b+p$ modulo 4, we have
$$
\deg y_{3}(z)=\max\{\deg y_{31}(z),\deg y_{32}(z)\}.
$$

Furthermore we suppose that $p\equiv 3$ (mod 8). Then we can define the following for non-negative integers $\alpha,\beta,u$ and $v$:
$$
\frac{1}{8}p=\alpha+\frac{3}{8},\quad \frac{3}{8}p=\beta+\frac{1}{8},\quad
\frac{5}{8}p=u+\frac{7}{8},\quad \frac{7}{8}p=v+\frac{5}{8}.
$$

Applying Lemma \ref{fn} by setting $m=3$ and $n=1$, we obtain
\begin{align*}
&y_{32}(z)\\
&= \sum\limits_{a=1}^{\beta-1}\left(\frac{a}{p}\right)z^{8a+p}+\left(\frac{\beta}{p}\right)\sum\limits_{a=0}^{p-2}(-1)^{a+1}z^{4a+3}-
\sum\limits_{a=\beta+1}^{v}\left(\frac{a}{p}\right)z^{8a-3p}+\sum\limits_{a=v+1}^{p-1}\left(\frac{a}{p}\right)z^{8a-7p}.
\end{align*}

It is easy to see that in the above summation any integer power of $z$ can appear at most two times.
Let $b$ be the integer $\frac{3p-1}{4}$.
Note that $\beta+1\le b\le v$, $\left(\frac{b}{p}\right)=-1$, and  $8b-3p=3p-2$. So there is one term $-\left(\frac{b}{p}\right)z^{3p-2}=z^{3p-2}$ in the above summation. In addition,
since $\left(\frac{8}{p}\right)\left(\frac{\beta}{p}\right)=-1$, we have $\left(\frac{\beta}{p}\right)=1$. Choose
 $a=\frac{3p-5}{4}$, we have $a \le p-2$ and $4a+3=3p-2$, note that $a$ is odd by the choice of $p$, so there is another one term $\left(\frac{\beta}{p}\right)z^{3p-2}=z^{3p-2}$ in the above summation. Thus, $y_{32}(z)$ has one term $2z^{3p-2}$. So, we have $\deg y_{32}(z)\ge 3p-2$, then $\deg y_{3}(z)\ge 3p-2$,
  which implies that
\begin{equation}\label{2^mp^n38}
\rho(t,r,q)\ge (3p-2)/(2(p-1))> 1.5.
\end{equation}

Finally we suppose that $p\equiv 7$ (mod 8) under the assumption $D=2p$.
As before, we define the following for non-negative integers $\alpha,\beta,u$ and $v$:
$$
\frac{1}{8}p=\alpha+\frac{7}{8},\quad \frac{3}{8}p=\beta+\frac{5}{8},\quad
\frac{5}{8}p=u+\frac{3}{8},\quad \frac{7}{8}p=v+\frac{1}{8}.
$$

Applying Lemma \ref{fn} by setting $m=3$ and $n=1$, we obtain
\begin{align*}
&y_{32}(z)\\
&= \sum\limits_{a=1}^{\beta}\left(\frac{a}{p}\right)z^{8a+p}-\sum\limits_{a=\beta+1}^{v-1}\left(\frac{a}{p}\right)z^{8a-3p}+
\left(\frac{v}{p}\right)\sum\limits_{a=0}^{p-2}(-1)^{a}z^{4a+3}+\sum\limits_{a=v+1}^{p-1}\left(\frac{a}{p}\right)z^{8a-7p}.
\end{align*}

It is easy to see that in the above summation any integer power of $z$ can appear at most two times.
Let $b$ be the integer $\frac{p+1}{4}$. Note that $1\le b\le \beta$, $\left(\frac{b}{p}\right)=1$, and  $8b+p=3p+2$. So there is one term $\left(\frac{b}{p}\right)z^{3p+2}=z^{3p+2}$ in the above summation. In addition,
since $\left(\frac{8}{p}\right)\left(\frac{v}{p}\right)=-1$, we have $\left(\frac{v}{p}\right)=-1$. Choose
 $a=\frac{3p-1}{4}$, we have $a\le p-2$ and $4a+3=3p+2$, note that $a$ is odd by the choice of $p$, so there is another term $-\left(\frac{v}{p}\right)z^{3p+2}=z^{3p+2}$ in the above summation. Thus, $y_{32}(z)$ has one term $2z^{3p+2}$. So, we have $\deg y_{32}(z)\ge 3p+2$, then $\deg y_{3}(z)\ge 3p+2$,
  which implies that
\begin{equation}\label{2^mp^n78}
\rho(t,r,q)\ge (3p+2)/(2(p-1))> 1.5.
\end{equation}

Therefore, we complete the proof of Proposition \ref{2^mp^n3}.
\end{proof}

\begin{proposition}\label{2^mp^n1}
For any odd prime $p$ with $p\equiv 1$ {\rm (mod 4)} and integers, $m\ge 2$ and $n\ge 1$,
Theorem $\ref{Thm1}$ $(\ref{2^mp^n0})$ is true for $k=2^mp^n$.
\end{proposition}
\begin{proof}
By \eqref{Dk}, if $m=2$, then $D=1$ or $D=p$; otherwise if $m\ge 3$, then $D=1,2,p$, or $2p$.

For the cases $D=1,2$, we can apply the same argument as the proof of
 Proposition \ref{2^mp^n3} to verify the desired result.

Now suppose that $D=p$. Since $p\equiv 1$ (mod 4), we have $\sqrt{p}=\sum\limits_{a=1}^{p-1}\left(\frac{a}{p}\right)\zeta_{p}^a=\sum\limits_{a=1}^{p-1}\left(\frac{a}{p}\right)\zeta_{k}^{a2^{m}p^{n-1}}$, then
$$
\sqrt{-p}=\sqrt{-1}\cdot\sqrt{p}=\sum\limits_{a=1}^{p-1}\left(\frac{a}{p}\right)\zeta_{k}^{2^{m-2}p^{n-1}(4a+p)}.
$$
Then
$$
y_{1}(x)\equiv \sum\limits_{a=1}^{p-1}\left(\frac{a}{p}\right)x^{2^{m-2}p^{n-1}(4a+p)+g}-\sum\limits_{a=1}^{p-1}\left(\frac{a}{p}\right)x^{2^{m-2}p^{n-1}(4a+p)} \quad
\textrm{(mod $r(x)$)}.
$$
Put
$$
y_{2}(x)\equiv \sum\limits_{a=1}^{p-1}\left(\frac{a}{p}\right)x^{2^{m-2}p^{n-1}(4a+p)} \quad \textrm{(mod $r(x)$)\quad with $\deg y_{2}(x)<\deg r(x)$}.
$$
Notice that $r(x)$ is also a polynomial with respect to $x^2$, and that every $2^{m-2}p^{n-1}(4a+p)+g$ and $2^{m-2}p^{n-1}(4b+p)$ have different parities for any two integers $a,b$.
So as before, we have
$$\deg y_{1}(x)\ge \deg y_{2}(x).$$
Then, we only need to consider $y_{2}(x)$. Put $z=x^{2^{m-2}p^{n-1}}$, then $r(x)=\Phi_{4p}(z)$.
Define
$$
y_{3}(z)\equiv \sum\limits_{a=1}^{p-1}\left(\frac{a}{p}\right)z^{4a+p} \quad \textrm{(mod $\Phi_{4p}(z)$)\quad with $\deg y_{3}(z)<\deg \Phi_{4p}(z)$}.
$$
Clearly, we have
 $$\deg y_{2}(x)=2^{m-2}p^{n-1}\deg y_{3}(z).$$

Assume that $p=4b+1$. Applying Lemma \ref{fn} by setting $m=2$ and $n=1$, and noticing that $\left(\frac{a}{p}\right)=\left(\frac{4a}{p}\right)$ for any integer $a$, we have
\begin{align*}
y_{3}(z)
&= \sum\limits_{a=1}^{b-1}\left(\frac{a}{p}\right)z^{4a+p}+\sum\limits_{a=0}^{p-2}(-1)^{a+1}z^{2a+1}-
\sum\limits_{a=b+1}^{3b}\left(\frac{a}{p}\right)z^{4a-p}+\sum\limits_{a=3b+1}^{p-1}\left(\frac{a}{p}\right)z^{4a-3p}\\
&= \sum\limits_{a=1}^{b-1}\left(\frac{a}{p}\right)z^{2(2a+2b)+1}+\sum\limits_{a=0}^{p-2}(-1)^{a+1}z^{2a+1}-
\sum\limits_{a=b+1}^{3b}\left(\frac{a}{p}\right)z^{2(2a-2b-1)+1}\\
   &\quad +\sum\limits_{a=3b+1}^{p-1}\left(\frac{a}{p}\right)z^{2(2a-6b-2)+1}\\
&= \sum\limits_{\substack{a=2(b+1)\\ \textrm{$a$ even}}}^{p-3}\left(\frac{2a+1}{p}\right)z^{2a+1}+\sum\limits_{a=0}^{p-2}(-1)^{a+1}z^{2a+1}-
\sum\limits_{\substack{a=1\\ \textrm{$a$ odd}}}^{p-2}\left(\frac{2a+1}{p}\right)z^{2a+1}\\
   &\quad +\sum\limits_{\substack{a=0\\ \textrm{$a$ even}}}^{2(b-1)}\left(\frac{2a+1}{p}\right)z^{2a+1},\\
&=\sum\limits_{a=0}^{p-2}(-1)^{a}(\left(\frac{2a+1}{p}\right)-1)z^{2a+1}.
\end{align*}

We first let $p>5$. By the choice of $p$, we have $p\ge 13$. Note that since $p\equiv 1$ {\rm (mod 4)}, we have $\left(\frac{-a}{p}\right)=\left(\frac{a}{p}\right)$ for any integer $a$. Now we claim that there exists an integer $\frac{3}{4}(p-1)-1\le a\le p-2$ such that $\left(\frac{2a+1}{p}\right)=-1$. Indeed, assume that for any integer $\frac{3}{4}(p-1)-1\le a\le p-2$, we have  $\left(\frac{2a+1}{p}\right)=1$, and thus $\left(\frac{-2a-1}{p}\right)=1$. Using proof by contradiction, we can find that $2a+1$ is not congruent to $-2c-1$ modulo $p$ for any $\frac{3}{4}(p-1)-1\le a,c\le p-2$. Then, the set $\{2a+1,-2a-1:\frac{3}{4}(p-1)-1\le a\le p-2\}$ is contained in some complete set $S$ of representatives for $\{a: 1\le a\le p-1\}$ (mod $p$) and its cardinality is $\frac{p+3}{2}$. So, there exist $\frac{p+3}{2}$ elements $a\in S$ with $\left(\frac{a}{p}\right)=1$. However, we know that there are exactly $\frac{p-1}{2}$ elements $a\in S$ with $\left(\frac{a}{p}\right)=1$. So the claim is proved. Thus, $y_{3}(z)$ has one term $(-1)^{a+1}2z^{2a+1}$
with $\frac{3}{4}(p-1)-1\le a\le p-2$. So, we have $\deg y_{3}(z)\ge \frac{3}{2}(p-1)-1$,
which implies that
\begin{equation}\label{2^mp^np}
\rho(t,r,q)\ge (\frac{3}{2}(p-1)-1)/(p-1)>1.4.
\end{equation}
Now we let $p=5$. Then we directly obtain $y_{3}(z)=2z^7-z^5+2z^3$, which implies that
\begin{equation}\label{2^mp^n50}
\rho(t,r,q)\ge 7/4=1.75.
\end{equation}

Now assume that $D=2p$. Since $\sqrt{-2}=\zeta_{8}^{3}+\zeta_8=\zeta_{k}^{3\cdot 2^{m-3}p^n}+\zeta_{k}^{2^{m-3}p^n}$, we have
$$
\sqrt{-D}=\sqrt{-2}\cdot \sqrt{p}=\sum\limits_{a=1}^{p-1}\left(\frac{a}{p}\right)\zeta_{k}^{2^{m-3}p^{n-1}(8a+3p)}+
\sum\limits_{a=1}^{p-1}\left(\frac{a}{p}\right)\zeta_{k}^{2^{m-3}p^{n-1}(8a+p)}.
$$
Then
$$
y_{1}(x)\equiv (x^g-1)\left[\sum\limits_{a=1}^{p-1}\left(\frac{a}{p}\right)x^{2^{m-3}p^{n-1}(8a+3p)}+\sum\limits_{a=1}^{p-1}\left(\frac{a}{p}\right)x^{2^{m-3}p^{n-1}(8a+p)}\right] \quad
\textrm{(mod $r(x)$)}.
$$
Put
$$
y_{2}(x)\equiv \sum\limits_{a=1}^{p-1}\left(\frac{a}{p}\right)x^{2^{m-3}p^{n-1}(8a+3p)}+\sum\limits_{a=1}^{p-1}\left(\frac{a}{p}\right)x^{2^{m-3}p^{n-1}(8a+p)} \quad
\textrm{(mod $r(x)$)}
$$
with $\deg y_{2}(x)< \deg r(x)$.
Similarly as before, we have
$$\deg y_{1}(x)\ge \deg y_{2}(x).$$
 Then we only need to consider $y_{2}(x)$.
 Set $z=x^{2^{m-3}p^{n-1}}$, then $r(x)=\Phi_{8p}(z)$. Define
$$
y_{3}(z)\equiv \sum\limits_{a=1}^{p-1}\left(\frac{a}{p}\right)z^{8a+3p}+\sum\limits_{a=1}^{p-1}\left(\frac{a}{p}\right)z^{8a+p} \quad \textrm{(mod $\Phi_{8p}(z)$)\quad with $\deg y_{3}(z)<\deg \Phi_{8p}(z)$}.
$$
Obviously, we have
$$\deg y_{2}(x)=2^{m-3}p^{n-1}\deg y_{3}(z).$$
 Moreover, define
\begin{align*}
&y_{31}(z)\equiv \sum\limits_{a=1}^{p-1}\left(\frac{a}{p}\right)z^{8a+3p} \quad \textrm{(mod $\Phi_{8p}(z)$)\quad with $\deg y_{31}(z)<\deg \Phi_{8p}(z)$},\\
&y_{32}(z)\equiv \sum\limits_{a=1}^{p-1}\left(\frac{a}{p}\right)z^{8a+p} \quad \textrm{(mod $\Phi_{8p}(z)$)\quad with $\deg y_{32}(z)<\deg \Phi_{8p}(z)$}.
\end{align*}
Then $y_{3}(z)=y_{31}(z)+y_{32}(z)$. Since $\Phi_{8p}(z)$ is also a polynomial with respect to $z^{4}$, and $8a+3p$ is not congruent to $8b+p$ modulo 4 for any two integers $a$ and $b$, we have
$$
\deg y_{3}(z)=\max\{\deg y_{31}(z),\deg y_{32}(z)\}.
$$

Furthermore we suppose that $p\equiv 1$ (mod 8). Then we can define the following for non-negative integers $\alpha,\beta,u$ and $v$,
$$
\frac{1}{8}p=\alpha+\frac{1}{8},\quad \frac{3}{8}p=\beta+\frac{3}{8},\quad
\frac{5}{8}p=u+\frac{5}{8},\quad \frac{7}{8}p=v+\frac{7}{8}.
$$

Applying Lemma \ref{fn} by setting $m=3$ and $n=1$, and
noticing that $\left(\frac{\alpha}{p}\right)=1$, we obtain
\begin{align*}
&y_{31}(z)\\
&= \sum\limits_{a=1}^{\alpha-1}\left(\frac{a}{p}\right)z^{8a+3p}+\sum\limits_{a=0}^{p-2}(-1)^{a+1}z^{4a+3}-
\sum\limits_{a=\alpha+1}^{u}\left(\frac{a}{p}\right)z^{8a-p}+\sum\limits_{a=u+1}^{p-1}\left(\frac{a}{p}\right)z^{8a-5p}.
\end{align*}

 By the choice of $p$, we find that $c=\frac{3}{4}(p-1)$ is an even integer, $c\le p-2$,
 and $4c+3=3p$. So, there is one term $(-1)^{c+1}z^{4c+3}=-z^{3p}$ in the above summation.
 Notice that the three integers $8a+3p, 8a-p, 8a-5p$ are not equal to $3p$ for any $1\le a\le p-1$.
 Thus, $y_{31}(z)$ has one term $-z^{3p}$.
 So, we have $\deg y_{31}(z)\ge 3p$. Then $\deg y_{3}(z)\ge 3p$, which implies that
\begin{equation}\label{2^mp^n81}
\rho(t,r,q)\ge 3p/(2(p-1))>1.5.
\end{equation}

Finally we suppose that $p\equiv 5$ (mod 8) under the assumption $D=2p$. As before we define the following for non-negative integers $\alpha,\beta,u$ and $v$:
$$
\frac{1}{8}p=\alpha+\frac{5}{8},\quad \frac{3}{8}p=\beta+\frac{7}{8},\quad
\frac{5}{8}p=u+\frac{1}{8},\quad \frac{7}{8}p=v+\frac{3}{8}.
$$

Applying Lemma \ref{fn} by setting $m=3$ and $n=1$, 
and noticing that $\left(\frac{u}{p}\right)=-1$, we obtain
\begin{align*}
&y_{31}(z)\\
&= \sum\limits_{a=1}^{\alpha}\left(\frac{a}{p}\right)z^{8a+3p}-
\sum\limits_{a=\alpha+1}^{u-1}\left(\frac{a}{p}\right)z^{8a-p}-
\sum\limits_{a=0}^{p-2}(-1)^{a}z^{4a+3}+\sum\limits_{a=u+1}^{p-1}\left(\frac{a}{p}\right)z^{8a-5p}.
\end{align*}

 By the choice of $p$, we find that $c=\frac{3}{4}(p-1)$ is an odd integer, $c\le p-2$,
 and $4c+3=3p$. So, there is one term $(-1)^{c+1}z^{4c+3}=z^{3p}$ in the above summation.
 Notice that the three integers $8a+3p, 8a-p, 8a-5p$ are not equal to $3p$ for any $1\le a\le p-1$.
 Thus, $y_{31}(z)$ has one term $z^{3p}$.
 So we have $\deg y_{31}(z)\ge 3p$. Then $\deg y_{3}(z)\ge 3p$, which implies that
\begin{equation}\label{2^mp^n81}
\rho(t,r,q)\ge 3p/(2(p-1))>1.5.
\end{equation}

This completes the proof of Proposition \ref{2^mp^n1}.
\end{proof}

\section{Numerical data}\label{Proof2}

In this section, we use PARI/GP to test the cyclotomic families in Theorem \ref{Thm1} with
embedding degree $k\le 82$.
Without presenting the source codes, we want to explain briefly how we can achieve this.

Here, we use the notation at the beginning of Section \ref{Proof1}.
To construct the cyclotomic families with embedding degree $k$ and CM discriminant $D$
considered in Theorem \ref{Thm1}, first calculate $r(x)=\Phi_k(x)$, and then compute
\begin{equation*}
t(x)\equiv x^g+1 \quad \textrm{(mod $r(x)$)}
\end{equation*}
 for $1\le g\le k$ and $\gcd(g,k)=1$.
In our cases, if $g$ is given, then $s(x)$ has an explicit formula, see Section \ref{Proof1}.
But here we use a uniform way to obtain $s(x)$, that is to find a root of $X^2+D$ in the
field $\Q[x]/(r(x))$, which can be done in PARI/GP, see \cite{Kang2007} for more details.
So, for any given $g$, following Theorem \ref{Brezing-Weng},
we can construct polynomials $t(x), y(x), q(x)$ explicitly.
Note that $t(x)$ automatically represents integers.
To test whether $(t(x),r(x),q(x))$ is a cyclotomic
family, we need to test whether $q(x)$ represents primes and whether $y(x)$ represents integers.
Letting $g$ run through all the possibilities, we obtain all the corresponding cyclotomic families,
and then we can easily get the smallest $\rho$-value and determine
whether the smallest possible $\rho$-value indicated in Theorem \ref{Thm1} can be achieved.

Now, we explain how to test whether a given polynomial $f(x)\in \Q[x]$ represents integers or
represents primes, actually it is easy and it has also been discussed below Definition 2.5 of \cite{Freeman2010}.

For a given polynomial $f(x)\in\Q[x]$, let $N$ be the least common multiple of the
denominators of its coefficients. To test whether it represents integers,
we only need to calculate $f(n)$ for $0\le n<N$.

 Now, to test whether $f(x)$ represents primes, by Definition \ref{reprime}, we only need to explain how to test
 $\gcd(\{f(x)\in\Z: x\in \Z\})=1$.
 Assume that $f(x)$ represents integers,
 compute $M=\gcd(\{f(n)\in\Z: 0\le n<N\})$, and 
suppose that $M\ne 1$. Then for every prime factor $p$ of $M$,
to determine whether $p|\gcd(\{f(x)\in\Z: x\in\Z\})$,
we only need to test whether $f(n)$ is divisible by $p$ for all $0\le n<pN$
when $f(n)\in\Z$. If $p\nmid\gcd(\{f(x)\in\Z: x\in\Z\})$ for every prime factor $p$ of $M$,
then we have $\gcd(\{f(x)\in\Z: x\in\Z\})=1$.

Using the above discussions, we can write some programs in PARI/GP to obtain
Table \ref{data}. Here, we omit the details.

For each $k\le 82$ with the form $k=2^mp^n$, Table \ref{data} gives the parameters $(\rho_k,D,g,\deg t(x))$ of a cyclotomic family in Theorem \ref{Thm1} with smallest $\rho$-value $\rho_k$. For example, given $k$ and $(\rho_k,D,g, \deg t(x))$, one can use the above discussions to construct a cyclotomic family $(t(x),r(x),q(x))$ with embedding degree $k$, CM discriminant $D$ and $\rho$-value $\rho_k$. Notice that for a given $k$, there may exist several cyclotomic families with smallest $\rho$-value $\rho_k$, here among them we choose a family such that $t(x)$ has smallest degree. We also want to indicate that for some $k$ there are no such cyclotomic families with embedding degree $k$.

\begin{table}
\centering
\caption{Cyclotomic families in Theorem \ref{Thm1} with $k\le 82$}
\label{data}
\begin{tabular}{|c|c||c|c|}
\hline
$k$ & $(\rho_k,D,g,\deg t(x))$ & $k$ & $(\rho_k,D,g,\deg t(x))$   \\ \hline

$1$ & No such cyclotomic families. & $37$  & No such cyclotomic families.   \\ \hline

$2$ & No such cyclotomic families. &  $38$ & $(1.889,19,3,3)$   \\ \hline

$3$ & No such cyclotomic families. &   $40$ &  $(1.750,5,11,11)$  \\ \hline

$4$ & No such cyclotomic families. &   $41$ &  No such cyclotomic families.   \\ \hline

$5$ & No such cyclotomic families. &  $43$ & $(1.890,43,5,5)$    \\ \hline

$6$ & No such cyclotomic families. &  $44$ & $(1.200,1,1,1)$   \\ \hline

$7$ & $(1.667,7,1,1)$ &  $46$  & $(1.727,23,1,1)$  \\ \hline

$8$ & No such cyclotomic families. &  $47$ & $(1.783,47,5,5)$   \\ \hline

$9$ & $(1.333,3,1,1)$ & $48$ &  $(1.125,3,1,1)$   \\ \hline

$10$ & No such cyclotomic families. &  $49$ & $(1.381,7,1,1)$  \\ \hline

$11$ & $(1.600,11,4,4)$ & $50$ & No such cyclotomic families.  \\ \hline

$12$ & $(1.500,3,1,1)$ &  $52$ & $(1.167,1,1,1)$  \\ \hline

$13$ & No such cyclotomic families. &  $53$ & No such cyclotomic families.  \\ \hline

$14$ & $(1.667,7,1,1)$ &  $54$ & No such cyclotomic families. \\ \hline

$16$ & No such cyclotomic families. & $56$ & $(1.417,7,1,1)$ \\ \hline

$17$ & No such cyclotomic families. &  $58$ & No such cyclotomic families. \\ \hline

$18$ & No such cyclotomic families. &  $59$ & $(1.828,59,16,16)$  \\ \hline

$19$ & $(1.667,19,10,10)$ & $61$ &  No such cyclotomic families. \\ \hline

$20$ & $(1.500,1,1,1)$ &  $62$ & $(1.867,31,7,7)$ \\ \hline

$22$ & $(1.800,11,3,3)$ &  $64$ & $(1.563,2,1,1)$ \\ \hline

$23$ & $(1.727,23,1,1)$ &  $67$ & $(1.848,67,25,25)$ \\ \hline

$24$ & $(1.250,3,1,1)$ &  $68$ & $(1.125,1,1,1)$\\ \hline

$25$ & No such cyclotomic families. &  $71$ & $(1.886,71,2,2)$  \\ \hline

$26$ & No such cyclotomic families. & $72$ & $(1.333,2,1,1)$ \\ \hline

$27$ & $(1.111,3,1,1)$  &   $73$ & No such cyclotomic families.  \\ \hline

$28$ & $(1.333,1,1,1)$ &   $74$ & No such cyclotomic families. \\ \hline

$29$ & No such cyclotomic families. &   $76$ & $(1.111,1,1,1)$  \\ \hline

$31$& $(1.667,31,3,3)$ &   $79$ & $(1.821,79,21,21)$  \\ \hline

$32$& $(1.625,2,1,1)$ &   $80$ & $(1.750,5,21,21)$  \\ \hline

$34$ & No such cyclotomic families. &   $81$ & $(1.037,3,1,1)$  \\ \hline

$36$ & $(1.667,1,1,1)$ &   $82$ & No such cyclotomic families.  \\ \hline

\end{tabular}
\end{table}

\section*{Acknowledgement}
The author would like to thank Keiji Okano for sending us his recent work \cite{Okano2012}.
 He also thanks the referee for careful reading and useful comments.

\end{document}